\documentclass[12pt]{amsart}

\newtheorem{thm}{Theorem}
\newtheorem{prop}[thm]{Proposition}

\def\ve{\varepsilon}

\begin{document}

\title{Two remarks on the Suita conjecture}

\author{Nikolai Nikolov}
\address{Institute of Mathematics and Informatics\\
Bulgarian Academy of Sciences\\
Acad. G. Bonchev Str., Block 8\\
1113 Sofia, Bulgaria}
\email{nik@math.bas.bg}
\address{Faculty of Information Sciences\\
State University of Library Studies
and Information Technologies\\
69A, Shipchenski prohod Str.\\
1574 Sofia, Bulgaria}

\date{}

\begin{abstract} It is shown that the weak multidimensional
Suita conjecture fails for any bounded non-pseudoconvex domain with
$C^{1+\ve}$-smooth boundary. On the other hand, it is proved that
the weak converse to the Suita conjecture holds for any finitely
connected planar domain.
\end{abstract}

\subjclass[2010]{32A25, 32F45, 32U35}

\keywords{Suita conjecture, Bergman kernel, Azukawa metric}

\maketitle

Let $D$ be a domain in $\Bbb C^n.$ Denote by $K_D$ and $A_D$ the
Bergman kernel on the diagonal and the Azukawa metric of $D$ (cf.
\cite{JP}). Let
$$I^A_D(z)=\{X\in\Bbb C^n:A_D(z;X)<1\}$$
be the indicatrix of $A_D$ at $z\in D.$

Z.~Blocki and W.~Zwonek have recently proved the following (see
\cite[Theorem 2]{BZ} and \cite[Theorem 7.5]{B2}).

\begin{thm} If $D$ is a pseudoconvex domain in $\Bbb C^n,$
then $$K_D(z)\ge\frac{1}{\lambda(I^A_D(z))},\quad z\in D.\footnote{If $\lambda(I^A_D(z))=\infty,$
then $K_D(z)=0.$}$$
\end{thm}

\bigskip

Theorem 1 for $n=1$ is known as the Suita conjecture (see
\cite{Sui}). The first proof of this conjecture is given in
\cite{B1}.

Theorem 1 can hold for some boun\-ded non-pseudoconvex domains.
To see this, note that if $M$ is a
closed pluripolar subset of a domain $D$ in $\Bbb C^n,$ then
$K_{D\setminus M}=K_D$ and $A_{D\setminus M}=A_D.$

On the other hand, our first remark says that even a weaker version of Theorem 1 fails
for bounded non-pseudoconvex domains with $C^{1+\ve}$-smooth
boundaries.

\begin{prop} Let $D$ is a bounded non-pseudoconvex domain
in $\Bbb C^n$ with $C^{1+\ve}$-smooth boundary ($\ve>0$). Then there
exists a sequence $(z_j)_j\subset D$ such that
$$\lim_{j\to\infty}K_D(z_j)\lambda(I^A_D(z_j))=0.$$
\end{prop}

\begin{proof} Since $D$ is non-pseudoconvex, we may find a sequence
$(z_j)_j\subset D$ approaching $a\in\partial D$ such that
\begin{equation}\label{ker}
\lim_{j\to\infty}K_D(z_j)<\infty
\end{equation}
(otherwise, $\log K_D$ would be a plurisubharmonic exhaustion
function for $D$). On the other hand, if $\ve\le 1,$ then, by
\cite[Proposition 2 (i)]{DNT}, there exists a constant $c_1>0$ such that
$$c_1A_D(z;X)\ge\frac{|X_N|}{(d_D(z))^{\frac{\ve}{1+\ve}}}+||X||,
\quad z\mbox{ near }a,$$ where $d_D(z)=\mbox{dist}(z,\partial D)$
and $X_N$ is the projection of $X$ on the complex normal to $\partial
D$ at a point $a'$ such that $||z-a'||=d_D(z).$
Thus, one may find a constant $c_2>0$ for which
$$\lambda(I^A_D(z))\le c_2(d_D(z))^{\frac{2\ve}{1+\ve}},\quad z\mbox{ near
}a.$$ This inequality and \eqref{ker} imply the wanted result.
\end{proof}

\begin{prop}\cite[Proposition 4]{BZ} Let $0<r<1$ and $P_r=\{z\in\mathbb
C:r<|z|<1\}.$ Then
$$K_{P_r}(\sqrt r)\ge-\frac{2\log
r}{\pi^2}\cdot\frac{1}{\lambda(I^A_{P_r}(\sqrt r))}.$$
\end{prop}

So, the converse to the Suita conjecture is not true with any
universal constant instead of 1. However, our second remark says
that any finitely connected planar has its own constant.

\begin{prop} For any finitely connected planar domain $D$
there exists a constat $c>0$ such that
$$K_D(z)\le\frac{c}{\lambda(I^A_D(z))},\quad z\in D.$$
\end{prop}

\begin{proof} It follows by the removable singularity theorem and the uniformization
theorem that it is enough to consider the case when $D=\Bbb D\setminus E,$ where
$E$ is a union of disjoint closed discs which belong to the open unit disc $\Bbb D.$
Since now $\partial D$ is $C^1$-smooth, then \cite[Proposition 2]{JN} (see also
\cite[Lemma 20.3.1]{JP}) and the remark at the end of \cite{JN} show that
$$\lim_{z\to\partial D}\frac{d^2_D(z)}{\lambda(I^A_D(z))}=
\frac{1}{4\pi}=\lim_{z\to\partial D}K_D(z)d^2_D(z).$$
Hence $$\lim_{z\to\partial D}K_D(z){\lambda(I^A_D(z))}=1$$
which leads to the desired inequality.
\end{proof}

{}

\end{document}